\documentclass[11pt]{amsart}
   \usepackage[margin=3.6cm]{geometry}

\numberwithin{equation}{section}
\usepackage{amsmath,amsfonts,amsthm,amssymb,amscd,mathabx, verbatim,graphicx,color,multirow,booktabs, caption,tikz,tikz-cd, mathdots,bm}
\usepackage{tikz-cd}
 \usepackage[pagebackref]{hyperref} 
\usetikzlibrary{positioning}

\newtheorem{theorem}{Theorem}[section]
\newtheorem{lemma}[theorem]{Lemma}

\newtheorem{proposition}[theorem]{Proposition}
 \newtheorem{corollary}[theorem]{Corollary}

      \newtheorem{conjecture}{Conjecture}

      \theoremstyle{definition}

     \theoremstyle{remark}
     \newtheorem{remark}[theorem]{Remark}

\newcommand{\Sym}{\mathop{\mathrm{Sym}}}
\newcommand{\Alt}{\mathop{\mathrm{Alt}}}

\newcommand{\Syl}{\mathop{\mathrm{Syl}}}
\newcommand{\Sylp}{\mathop{\mathrm{Syl}_p}}
\newcommand{\GL}{\mathop{\mathrm{GL}}}

\newcommand{\Core}{\mathop{\mathrm{Core}}}
\newcommand{\CoreG}{\mathop{\mathrm{Core}_G}}

\newcommand{\DS}{\mathcal{DS}}

\newcommand\blfootnote[1]{%
  \begingroup
  \renewcommand\thefootnote{}\footnote{#1}%
  \addtocounter{footnote}{-1}%
  \endgroup
}

 \definecolor{mycolor}{rgb}{0.55,0.0,0.16}
  \definecolor{myred}{rgb}{0.75,0.0,0.16} 
  \definecolor{mygreen}{rgb}{0.0,0.4,0.16} 
  \definecolor{myviolet}{rgb}{1,0,1} 
   \definecolor{mypink}{rgb}{0.67,0,0.47}
	
	\newcommand{\MYhref}[3][black]{\href{#2}{\color{#1}{#3}}}%
	
\hypersetup{
colorlinks=true,
linkcolor=true,
linktocpage=true,
pageanchor=true,
hyperindex=true
}

 \AtBeginDocument{
     \hypersetup{
  linkcolor=mycolor,
  urlcolor=mypink,
citecolor=mygreen
}
     }

\makeatletter
\@namedef{subjclassname@2020}{%
  \textup{2020} Mathematics Subject Classification}
\makeatother

\subjclass[2020]{Primary: 20D20, 20B30, 20Fxx}  
\keywords{Sylow subgroups, nilpotent subgroups, symmetric group, synchronization} 

\author[Francesca Lisi]{Francesca Lisi}
\address{\parbox{\linewidth}{Francesca Lisi}}
\email{francesca.lisi@postecert.it} 

\author[Luca Sabatini]{Luca Sabatini}
\address{\parbox{\linewidth}{Luca Sabatini, Mathematics Institute, Zeeman Building, University of Warwick\\
Coventry CV4\,7AL, United Kingdom \vspace{0.1cm}}}
\email{luca.sabatini@warwick.ac.uk, sabatini.math@gmail.com} 

 \title[Sylow subgroups for distinct primes]{Sylow subgroups for distinct primes\\and intersection of nilpotent subgroups} 

\begin{document}

\maketitle 

\blfootnote{This paper is to the memory of Marty Isaacs and the grace of his mathematics.}

\begin{abstract} 
Let $G$ be a finite group and let $(P_i)_{i=1}^n$ be Sylow subgroups for distinct primes $p_1,\ldots,p_n$.
We conjecture that there exists $x \in G$ such that $P_i \cap P_i^x$ is inclusion-minimal
in $\{ P_i \cap P_i^g : g \in G\}$ for all $i$.
As a first step in this direction,
we show that a finite group cannot be covered by (proper) Sylow normalizers for distinct primes.
Then we settle the conjecture in two opposite situations:
symmetric and alternating groups of large degree and metanilpotent groups of odd order.
Applications concerning the intersections of nilpotent subgroups are discussed.
\end{abstract}


\vspace{0.5cm}
\section{Introduction} \label{sec1}

In this paper we study the action of a finite group on the totality of its Sylow subgroups by conjugation.
This will provide a unified perspective on several results in the literature:
in particular, it relates theorems on the intersection of Sylow subgroups
\cite{Bro63,DGG+25,Ebe25,Ito58,Man75,MZ96,Rob84}
to theorems on the intersection of nilpotent subgroups
\cite{Bia75,Her71,Kur13,Pas66,Zen94,Zen96}.

Let $G$ be a finite group and let $\Syl_p(G)$ be the set of Sylow $p$-subgroups.
Let $\rho(G)$ be the set of primes $p$ for which $G$ has a non-normal Sylow $p$-subgroup and let $r=|\rho(G)|$.
Then $G$ acts on the Cartesian product
$$ \DS(G) \> = \> \bigtimes_{p \in \rho(G)} \Sylp(G) $$
by conjugation: If $P_i \in \Syl_{p_i}(G)$ for some $p_i \in \rho(G)$, and $g \in G$,
then
$$ (P_1,\ldots,P_r)^g \> = \> (P_1^g,\ldots,P_r^g) . $$

The kernel of this action is the intersection of all Sylow normalizers.
By an old theorem of Baer \cite{Bae53} this is the hypercenter
(i.e. the final term in the upper central series)
and the corresponding quotient is the permutation group on $\DS(G)$ that we are interested in.
If $r=1$, then Sylow's theorem implies that
the action of $G$ on $\DS(G)$ is transitive.
In general this is false, as for example it can happen that $N_G(P_1) \subseteq N_G(P_2)$, so there is no $g \in G$ such that $(P_1,P_2)^g = (P_1,\overline{P_2})$ if $P_2 \neq \overline{P_2}$.
The main purpose of the paper is to show that
it is still possible to move the elements of $\DS(G)$ significantly under the action of $G$
and to highlight some consequences of this fact.

We propose the following conjecture:

\begin{conjecture}[Main] \label{conj:A}
Let $G$ be a finite group and let $(P_i)_{i=1}^n$ be Sylow subgroups for distinct primes $p_1,\ldots,p_n$.
Then there exists $x \in G$ such that $P_i \cap P_i^x$ is minimal with respect to inclusion in $\{ P_i \cap P_i^g : g \in G\}$,
for all $i=1,\ldots,n$.
\end{conjecture} 

This is a ``synchronization'' statement,
where the crucial property of the element $x$ is that it does not depend on $i$.
It is easy to see that Conjecture \ref{conj:A} holds if $G$ acts transitively on $\DS(G)$.
For a prime $p$, the $p$-core $O_p(G)$ is the intersection of all the Sylow $p$-subgroups of $G$.
We say that a finite group $G$ has $(*)_p$
if there exist two Sylow $p$-subgroups with intersection in $O_p(G)$,
and that $G$ has $(*)$ if it has $(*)_p$ for all $p$.

\begin{conjecture}[Good case] \label{conj:B}
Let $G$ be a finite group with $(*)$ and let $(P_i)_{i=1}^n$ be Sylow subgroups for distinct primes $p_1,\ldots,p_n$.
Then there exists $x \in G$ such that $P_i \cap P_i^x = O_{p_i}(G)$, for all $i=1,\ldots,n$.
\end{conjecture} 

As explained in \cite[Th. B]{Zen96}, the class of groups with $(*)$ is quite rich.
In particular, all groups of odd order \cite{Ito58} and all finite simple groups \cite{MZ96} have $(*)$.
The typical group not having $(*)$ is $V \rtimes P$,
where $V=\mathbb{F}_3^2$ and $P \cong D_8$ is a $2$-subgroup of $\GL_2(3)$
(here $O_2(VP)=1$ but $P \cap P^v \neq 1$ for any $v \in V$).
In some sense Conjecture \ref{conj:B} is at the heart of Conjecture \ref{conj:A},
as the condition of minimality seems to be more of a technical issue.

The connection between Conjecture \ref{conj:A} and
the intersections of nilpotent subgroups is easy to explain.
If $H$ is a nilpotent subgroup of $G$, then $H = H_{p_1} \times \cdots \times H_{p_n}$,
where $H_{p_i}$ is a $p_i$-subgroup for each $i$.
If $x \in G$, note that $H^x = H_{p_1}^x \times \cdots \times H_{p_n}^x$, and
$$ H \cap H^x \> = \> 
(H_{p_1} \cap H_{p_1}^x) \times \cdots \times (H_{p_n} \cap H_{p_n}^x) . $$

The existence of an element $x$
as in Conjecture \ref{conj:A} can be used to settle a big portion of $H \cap H^x$ into the Fitting subgroup of $G$
(in particular $H \cap H^x \subseteq F(G)$ if and only if $H_{p_i} \cap H_{p_i}^x \subseteq O_{p_i}(G)$ for each $i$),
a problem which has been investigated by several authors including
Bialostocki, Dolfi, Herzog, Mann and Zenkov \cite{Bia75,Dol05,Her71,Man71,Zen94}
(see Section \ref{subSec2.2} for a summary of some of the known results).\\

We believe that proving Conjectures \ref{conj:A} and \ref{conj:B} would be marvelous,
but we have not been able to achieve this task.
However, we have made some progress in understanding the action of $G$ on $\DS(G)$.
First we prove that, in addition to the case $r=1$,
the action is transitive when only two primes divide the cardinality of the group.

\begin{theorem} \label{thTwoPrimes}
Let $G$ be a group of order $p^\alpha q^\beta$ for primes $p,q$ and $\alpha,\beta \geq 0$.
Then the action of $G$ on $\DS(G)$ is transitive.
\end{theorem}

This does not extend to three primes:
we can take $V \rtimes K$ where $V$ is a faithful irreducible $K$-module
and $K$ is nilpotent but not a $p$-group.
(Choosing two Sylow subgroups in the same complement $K$, both their normalizers will be $K$.)
It is natural to ask if a weaker property holds in general.
For example, given $(P_1,\ldots,P_r) \in \DS(G)$,
can we find an element $x \in G$ that moves all of the coordinates?
Note that this is equivalent to the condition $\bigcup_{i=1}^r N_G(P_i) \neq G$.

\begin{theorem} \label{thUnion} 
Let $G$ be a finite group and let $(P_i)_{i=1}^n$ be non-normal Sylow subgroups for distinct primes.
Then $\bigcup_{i=1}^n N_G(P_i) \neq G$.
\end{theorem} 

The hypothesis that the primes are distinct is essential:
$G = \Sym(4)$ can be realized as the union of all its Sylow normalizers.
The proof of Theorem \ref{thUnion}
is elementary
and exploits a nice lemma of Bryce, Fedri and Serena \cite{BFS95}.
The following corollary is immediate:

\begin{corollary}
If $|P_i:O_{p_i}(G)| \leq p_i$ for each prime $p_i$ and $P_i \in \Syl_{p_i}(G)$,
then there exists $x \in G$ such that $P_i \cap P_i^x = O_{p_i}(G)$ for all $i$.
\end{corollary} 

In Section \ref{sec4}, we establish Conjecture \ref{conj:A} in some special cases.
First, using a probabilistic argument, we deal with the symmetric and alternating groups of large degree.

\begin{theorem} \label{thSym}
There exists $n_0$ such that, if $G=\Sym(n)$ or $\Alt(n)$ for $n \geq n_0$,
and $(P_i)_{i \geq 1}$ are Sylow subgroups for distinct primes,
then there exists $x \in G$ such that $P_i \cap P_i^x =1$ for all $i$.
\end{theorem} 

To prove Theorem \ref{thSym}, we first obtain a refinement of \cite[Th. 1.1]{DGG+25} (Proposition \ref{propSym}).
This is achieved by using \cite{Ebe25} and some attention for odd primes.
Very recently, Burness and Huang \cite{BH25} proved Conjecture \ref{conj:A}
for all non-alternating simple groups,
completing the proof of a beautiful conjecture of Vdovin \cite[Prob. 15.40]{Kou22}:
if $G$ is simple and $H$ is a proper nilpotent subgroup,
then there exists $x \in G$ such that $H \cap H^x=1$.

Our final result establishes Conjecture \ref{conj:A} for metanilpotent groups of odd order.
This is a natural first place to look for a counterexample since the set of elements with the desired
property can be very small and the probabilistic approach is ineffective (see Remark \ref{remSmallGS}).
In order to apply an inductive argument, we prove a stronger statement that is false in the general case.

\begin{theorem} \label{thMetaNilp}
Let $G$ be a metanilpotent group of odd order
and let $(P_i)_{i=1}^n$ be Sylow subgroups for distinct primes $p_1,\ldots,p_n$.
Then there exists $x \in F(G)$ such that $P_i \cap P_i^x = O_{p_i}(G)$ for all $i$.
\end{theorem} 

The proof of Theorem \ref{thMetaNilp}
contains ideas that could be useful in a more general setting.
A key fact we use is that, if $P=Q^a$ and $Q \cap Q^b =1$ for some commuting elements $a$ and $b$,
then $P \cap P^b =1$.

\vspace{0.1cm}
\section{$\DS(G)$ and nilpotent subgroups} \label{sec2}

\subsection{The action on $\DS(G)$}

All groups in this paper are finite.
For a finite group $G$ and a set of primes $\pi$, $O_\pi(G)$ is the largest normal $\pi$-subgroup of $G$.
The Fitting subgroup $F(G) = \prod_p O_p(G)$ is the largest nilpotent normal subgroup,
and the hypercenter $\mathcal{H}(G)$ is the last term of the upper central series of $G$.
(Equivalently, it is the smallest normal subgroup modulo which the center is trivial.)
It is clear that $\mathcal{H}(G) \subseteq F(G)$.

\begin{theorem}[Baer] \label{thHyper} 
Let $G$ be a finite group. Then $\mathcal{H}(G)$ coincides with
\begin{itemize} 
    \item[(i)] the intersection of the maximal nilpotent subgroups of $G$;
    \item[(ii)] the intersection of all the Sylow normalizers in $G$;
    \item[(iii)] the kernel of the action of $G$ on $\DS(G)$.
\end{itemize}
\end{theorem}
\begin{proof}
See \cite[Cor. 3-4]{Bae53}.
Point (iii) follows from (ii).
\end{proof}

\begin{lemma} \label{lemRedHyp}
$F(G/\mathcal{H}(G)) = F(G)/\mathcal{H}(G)$.
\end{lemma}
\begin{proof}
We first notice that $F(G/Z(G))=F(G)/Z(G)$.
In fact, if $F_0/Z(G) = F(G/Z(G))$,
then $Z(G) \subseteq F_0 \cap C_G(F_0) = Z(F_0)$, and $F_0/Z(F_0)$ is nilpotent.
So $F_0$ itself is nilpotent, and $F_0=F(G)$.
The proof follows by the definition of $\mathcal{H}(G)$ and induction.
\end{proof}

In proving all the statements in this paper, we can always assume that we are working in $G/\mathcal{H}(G)$,
or equivalently that the (hyper)center of the group $G$ is trivial.

\begin{remark}
The action of $G$ on $\DS(G)$ is a special case of the action studied in the recent paper \cite{AMB24},
where the subgroups $H_i$'s are Sylow normalizers for distinct primes.
(Except that the $N_G(P_i)$'s are not core-free in general, but only the intersection of their cores is trivial.)
The relationship between the two articles ends here, because \cite{AMB24} goes in a different direction,
namely the search for regular orbits.
\end{remark}

\subsection{Intersection of nilpotent subgroups} \label{subSec2.2}

There is an extensive literature on the intersections of Sylow subgroups,
and nilpotent subgroups more generally, in finite groups.
There are also natural connections with other topics, such as the existence of regular orbits in linear groups
(see \cite{Har81}, for example).
 The latter connection arises from the following easy observation:

\begin{lemma} \label{lemIntCen}
Let $G=V \rtimes K$.
For each $v \in V$, we have $C_K(v) =K \cap K^v$.
\end{lemma} 
\begin{proof}
It is easy to see that $C_K(v) \subseteq K \cap K^v$.
Now suppose $a = b^v \in K \cap K^v$ for some $a,b \in K$ and $v \in V$.
Multiplying by $b^{-1}$ we obtain $b^{-1}a = b^{-1} v^{-1} b v$.
The left side lies in $K$, while the right side lies in $V$, so $a=b \in C_K(v)$. 
\end{proof} 

We now report two results that we briefly mentioned in Section \ref{sec1}.
Recall that a finite group $G$ has $(*)_p$
if there exist two Sylow $p$-subgroups with intersection in $O_p(G)$,
and that $G$ has $(*)$ if it has $(*)_p$ for all $p$.

\begin{theorem}[It\^o \cite{Ito58}] \label{thIto}
Every finite group of odd order has $(*)$.
\end{theorem} 

\begin{theorem}[Mazurov-Zenkov \cite{MZ96}] \label{thMZ}
Every finite simple group has $(*)$.
\end{theorem} 

See \cite{Bro63,Man75,Pas66,Rob84} for related work on the more specific property $(*)_p$.
In particular, Brodkey \cite{Bro63} showed that $(*)_p$
holds whenever the Sylow $p$-subgroups are abelian.

We now move to more general nilpotent subgroups.
Given a set of primes $\pi$ and a $\pi$-subgroup $H$ of $G$, we clearly have
$$ \CoreG(H) \subseteq O_\pi(G) . $$
If $H$ is nilpotent then
$$ \CoreG(H) \subseteq H \cap F(G) \subseteq O_\pi(F(G)) \subseteq O_\pi(G) . $$ 
If in addition $O_\pi(G) \subseteq H$, then
$$ \CoreG(H) = O_\pi(G) $$
and we have all equalities above.
The condition $O_\pi(G) \subseteq H$ is assured if $\pi=\{p\}$ and $H$ is a Sylow $p$-subgroup,
or if $G$ is solvable and $H$ is nilpotent Hall $\pi$-subgroup.
In fact, Bialostocki \cite{Bia75} generalized It\^o's theorem to nilpotent Hall subgroups of groups of odd order.

\begin{theorem}[Bialostocki] \label{thBialoHall}
Let $G$ be a group of odd order.
If $H \leqslant G$ is a nilpotent Hall $\pi$-subgroup, then there exists $x \in G$ such that $H \cap H^x = O_\pi(G)$.
\end{theorem}
\begin{proof}
This is \cite[Th. B.4]{Bia75}. See also \cite{Dol05}.
\end{proof} 

In general, even if $|G|$ is odd and $H$ is abelian but not a Hall subgroup,
it is false that $H \cap H^x = \Core_G(H)$ for some $x \in G$.
Indeed the number of conjugates of $H$ required to obtain the core can be arbitrarily large:
take $H_0$ abelian, $G= H_0 \wr_n C_n$,
and $H < H_0^n$ a core-free subgroup of $G$ isomorphic to $H_0^{n-1}$.
The intersection of any $b \geq 1$ conjugates of $H$ will contain a copy of $H_0^{n-b}$,
showing that $n$ conjugates are required.

There are more general results that provide $H \cap H^x \subseteq F(G)$,
 the most well-known being a theorem of Zenkov \cite{Zen94} on abelian subgroups.

\begin{theorem}[Zenkov] \label{thZenkovG}
Let $G$ be a finite group and let $A \leqslant G$ be abelian.
Then there exists $x \in G$ such that $A \cap A^x \subseteq F(G)$.
\end{theorem}
\begin{proof}
See \cite{Zen94}, or \cite[Th. 2.18]{Isa08}.
\end{proof}

In fact, Theorem \ref{thZenkovG} is the ``correct'' generalization of Brodkey's theorem,
rather than $H \cap H^x = \Core_G(H)$ as it might appear at first glance.
This observation played a crucial role in the formulation of Conjecture \ref{conj:A}.

The following is the equivalent of Vdovin's conjecture for groups of odd order:

\begin{conjecture} \label{conj:C}
Let $G$ be a group of odd order and let $H \leqslant G$ be nilpotent.
Then there exists $x \in G$ such that $H \cap H^x \subseteq F(G)$.
\end{conjecture}

Observe that Conjecture \ref{conj:A} $\Rightarrow$ Conjecture \ref{conj:B} $\Rightarrow$ Conjecture \ref{conj:C}.
We conclude this section with some partial results towards Conjecture \ref{conj:C}, starting with a theorem of Mann.
The proof at the end of \cite{Man75} contains some typos and is difficult to read,
so we rewrite it here for completeness.

\begin{lemma}[Mann] \label{lemInj}
Let $G$ be a group of odd order and let $H \leqslant G$ be nilpotent.
If $F(G) \subseteq H$, then there exists $x \in G$ such that $H \cap H^x =F(G)$.
\end{lemma}
\begin{proof} 
We can assume that $H$ is a maximal nilpotent subgroup containing $F(G)$.
Then, by \cite[Th. 1]{Man71}, $H$ is a nilpotent injector of $G$ (see \cite{Man71} for the definition of injector).
It is well known that any two injectors are conjugate, so it is sufficient to provide two injectors intersecting in $F(G)$.
For any prime $p$, let $C_p = C_G(O_{p'}(F(G)))$, and let $S_p \in \Syl_p(C_p)$.
We have $O_p(C_p) = O_p(G)$,
and it is proved in \cite{Man71} that $S= \prod_p S_p$ is a nilpotent injector of $G$.
For each $p$, by It\^o's theorem there exists $R_p \in \Syl_p(C_p)$ such that $S_p \cap R_p = O_p(G)$.
It follows that $R= \prod_p R_p$ is an injector of $G$ and $S \cap R=F(G)$.
\end{proof}

\begin{lemma} \label{lemImprov}
Let $G$ be a group of odd order and let $H < G$ be nilpotent.
\begin{itemize}
\item[(i)] If $H$ is maximal, then there exists $x \in G$ such that $H \cap H^x = \Core_G(H)$;
\item[(ii)] if $(|F(G)|,|H|) = 1$, then there exists $x \in G$ such that $H \cap H^x =1$.
\end{itemize}
\end{lemma}
\begin{proof}
(i) We can assume $\Core_G(H)=1$.
Let $V \unlhd G$ be a minimal normal subgroup, so that $V$ is an elementary abelian $p$-group for some prime $p$.
It is easy to see that $G=V \rtimes H$ and $C_H(V)=1$.
Then $H$ is a Hall subgroup of $G$ and the proof follows from Theorem \ref{thBialoHall}.

(ii) Let $G$ be a minimal counterexample.
Let $N=F(G)$, and observe that $N \subseteq F(NH)$.
If $N \neq F(NH)$, then there exists a prime $p$ not dividing $|N|$ such that $O_p(NH) \neq 1$.
On the other hand we have
$$ O_p(NH) = O_p(F(NH)) \subseteq C_G(N) \subseteq N , $$
and this is clearly impossible.
So $N=F(NH)$. As a consequence, if $NH \neq G$, then $NH$ gives a smaller counterexample, contradicting the minimality of $G$.
Otherwise $H$ is a Hall subgroup of $G$, and again we are done by Theorem \ref{thBialoHall}.
\end{proof}

\vspace{0.1cm}
\section{Theorems \ref{thTwoPrimes} and \ref{thUnion}} \label{sec3}

The following fact is well known.

\begin{lemma} \label{lemCosets}
Let $H,K \leqslant G$. The following hold:
\begin{itemize}
    \item[(i)] $|HK|/|H|$ is the number of right cosets of $H$ with non-empty intersection with $K$;
    \item[(ii)] if $H \cap Kg \neq \emptyset$ for some $g \in G$,
    then $H \cap Kg$ is a right coset of $H \cap K$. 
\end{itemize}
\end{lemma}
\begin{proof}
Part (i) is clear. For (ii), suppose $h \in H \cap Kg$.
Then
$$ H \cap Kg = H \cap Kh = (H \cap K)h . \qedhere $$
\end{proof}

\begin{lemma} \label{lemSyn}
Let $p_1,p_2$ be distinct primes and suppose that $G=N_G(P_1) N_G(P_2)$ where $P_i \in \Syl_{p_i}(G)$.
Then for all $Q_i \in \Syl_{p_i}(G)$ ($i=1,2$)
there are at least $|N_G(P_1) \cap N_G(P_2)|$ elements $x \in G$ such that $P_1^x = Q_1$ and $P_2^x= Q_2$.
\end{lemma}
\begin{proof}
Let $I=N_G(P_1)$, $J=N_G(P_2)$,
and $P_1^y=Q_1$, $P_2^z = Q_2$ for some $y,z \in G$.
Note that $(P_1)^{gy}=Q_1$ for every $g \in I$, and $(P_2)^{gz}=Q_2$ for every $g \in J$.
Then, all the elements of $Iy$ provide the desired conjugate of $P_1$,
and all the elements of $Jz$ provide the desired conjugate of $P_2$.
Now $|Iy \cap Jz| = |I \cap Jzy^{-1}|$.
By Lemma \ref{lemCosets} and $G=JI$,
we have that $I$ intersects every right coset of $J$, so this cardinality is nonzero.
In particular $Iy \cap Jz$ is non-empty and has cardinality $|I \cap J|$.
\end{proof}

\begin{proof}[Proof of Theorem \ref{thTwoPrimes}]
Immediate from Lemma \ref{lemSyn} (note that $G=P_1P_2$).
\end{proof}

We move to Theorem \ref{thUnion}.
A cover of proper subgroups $G = \bigcup_{i=1}^n H_i$ is {\itshape irredundant}
if removing any of the $H_i$'s does not give a cover anymore.
The next pretty result is \cite[Lem. 2.2]{BFS95}.

\begin{lemma}[Bryce, Fedri \& Serena] \label{lemBFS}
Let $G = \bigcup_{i=1}^n H_i$ be an irredundant cover.
If $p$ is a prime with $p \geq n$, then $\bigcap_{i=1}^n H_i$ contains all the $p$-elements of $G$.
\end{lemma}

\begin{proof}[Proof of Theorem \ref{thUnion}]
Let $G = \bigcup_{i=1}^n N_G(P_i)$ be a minimal counterexample with respect to $|G|$ and $n$.
Then $n \geq 2$ and the cover is irredundant by the minimality of $n$.
(Note that in general not all the primes corresponding to non-normal Sylow subgroups would appear in this union.)
If $p_1 < \ldots <p_n$ are the primes appearing, then it is clear that $p_n>n$.
By Lemma \ref{lemBFS}, there exists a prime $q$ dividing $|G|$
such that every $q$-element of $G$ is contained in $N_G(P_i)$ for all $i$.

Let $K \unlhd G$ be the subgroup generated by the $q$-elements and set $\overline{G} = G/K$.
It is clear that $\overline{P_i} = KP_i/K \in \Syl_{p_i}(\overline{G})$.
As a consequence of the correspondence theorem, it is easy to see that
$$ N_{\overline{G}}(\overline{P_i}) = \frac{N_G(KP_i)}{K} . $$
Moreover $N_G(P_i) \subseteq N_G(KP_i)$.
If $KP_i \unlhd G$ for some $i$, then
by the Frattini argument we would have $(K P_i) N_G(P_i) = G$.
Since $K \subseteq N_G(P_i)$, the left side is $N_G(P_i)$, so this is impossible because $P_i$ is non-normal.
It follows that $N_{\overline{G}}(\overline{P_i}) \neq \overline{G}$ for each $i$,
so that $\overline{G} = \bigcup_{i=1}^n N_{\overline{G}}(\overline{P_i})$ provides a smaller counterexample,
against the minimality of $G$.
\end{proof}

\vspace{0.1cm}
\section{Two synchronization theorems} \label{sec4} 

In this section we prove Conjecture \ref{conj:A} in some cases
where Theorems \ref{thTwoPrimes} and \ref{thUnion} are not helpful.
For a finite group $G$ and $P \in \Syl_p(G)$, let
$$ \Gamma_G(P) \> = \> \{ x \in G : P \cap P^x \mbox{ is inclusion-minimal in } \{ P \cap P^g : g \in G \} \} , $$
i.e. the set of ``good elements'' for $P$.
Note that $\Gamma_G(P)$ is a union of right cosets of $N_G(P)$.
When $(P_1,\ldots,P_r) \in \DS(G)$, let
$$ \Gamma_G(P_1,\ldots,P_r) \> = \> \bigcap_{i=1}^r \Gamma_G(P_i) . $$
Conjecture \ref{conj:A} asserts that $\Gamma_G(P_1,\ldots,P_r)$ is non-empty for all finite groups $G$
and any choice of $(P_1,\ldots,P_r)$.
Observe that $|\Gamma_G(P)|$ does not depend on the specific choice of $P \in \Syl_p(G)$,
whereas $|\Gamma_G(P_1,\ldots,P_r)|$ does depend on $(P_1,\ldots,P_r)$ in general.

\subsection{Symmetric and alternating groups} 

Let $G=\Sym(n)$ or $G=\Alt(n)$.
Observe that $O_p(G) =1$ for all $p$ when $n \geq 5$,
and that $G$ has $(*)$ when $n \geq 9$ \cite{MZ96}.
Let $n \geq 9$ and $(P_1,\ldots,P_r) \in \DS(G)$.
By the union bound we have
$$ Prob_{g \in G} ( P_i \cap P_i^g \neq 1 \mbox{ for some } i  ) 
\> \leq \>
\sum_{i=1}^r Prob_{g \in G} (P_i \cap P_i^g \neq 1) . $$
Therefore, to prove that $\Gamma_G(P_1,\ldots,P_r)$ is non-empty it is sufficient to show that
\begin{equation} \label{eqSym}
\sum_{i=1}^r Prob_{g \in G}(P_i \cap P_i^g \neq 1) \> < \> 1 .
\end{equation}
We will prove (\ref{eqSym}) for all large $n$,
using that $\Gamma_G(P_i)$ is very big for odd primes, and quite big when $p=2$.
This is a refinement of \cite[Th. 1.1]{DGG+25}.

\begin{proposition} \label{propSym}
If $G=\Sym(n)$ or $G=\Alt(n)$ and $P \in \Syl_p(G)$, then
\begin{itemize}
\item[(i)] If $p \neq 2$, then $Prob_{g \in G}(P \cap P^g \neq 1) = O \left( \frac{1}{n} \right)$;
\item[(ii)] If $p=2$, then $Prob_{g \in G}(P \cap P^g \neq 1) \leq 0.99$ for all large $n$.
\end{itemize}
\end{proposition}
\begin{proof}
We start with (i), so let $p \neq 2$.
We will work in $G=\Sym(n)$, but essentially the same proof is valid for $\Alt(n)$.
Let $x \in G$ be an element of order $p$ with less than $p$ fixed points.
As in \cite[Sec. 4]{DGG+25}, let $f'(n,p)$ be the probability that $x$ and a random conjugate $x^g$
both centralize a common element of order $p$.
We have
$$ Prob_{g \in G}(P \cap P^g \neq 1) \leq f'(n,p) . $$
By \cite[Lem. 4.7]{DGG+25}, we can assume that $n$ is a multiple of $p$,
so there exists $x \in G$ of order $p$ without fixed points.
Since $G$ has trivial center, it is easy to see that
$$ f'(n,p) \leq Prob_{g \in G}(\langle x,x^g \rangle \ngeqslant \Alt(n)) . $$
We now refer to \cite[Sec. 3-4]{EG21}
(in that paper set $c_i=c_i'=0$ for $i=1,2$, since both $x$ and $x^g$ have neither fixed points nor $2$-cycles).
By \cite[Sec. 4]{EG21}, the contribution of transitive subgroups not containing $\Alt(n)$ has exponential decay in $n$.
We continue as in \cite[Sec. 3]{EG21}.
By the Markov inequality, we have
$$ Prob_{g \in G}(\langle x,x^g \rangle \mbox{ is intransitive} ) \leq
\sum_{k = p,2p,\ldots,\tfrac{n}{p}} \mathbb{E}[N_k] , $$
where $N_k$ is the number of orbits of $\langle x,x^g \rangle$ of size $k$.
By \cite[Lem. 3.2-3.3]{EG21}, for each $k$ we can write
$$ \mathbb{E}[N_k] \leq \min \left\{ \left( \frac{k}{n} \right)^{k/6} e^{O(k)} , e^{-\Omega(k)} \right\} . $$
For $p \geq 7$ (and so $k \geq 7$), the proof follows by separating according to say $k=10 \log n$.
When $p=k=3$, $\mathbb{E}[N_3]$ is the expected number of $3$-subsets fixed both by $x$ and $x^g$.
It is easy to see that $\mathbb{E}[N_3] = O(n^{-1})$.
For $p=k=5$, we actually have $\mathbb{E}[N_5] = O(n^{-4})$.

For (ii), a stronger result is proven in \cite{Ebe25}.
In fact, $Prob_{g \in G}(P \cap P^g \neq 1)$ tends to $1-e^{-1/2} \approx 0.39$ when $G=\Sym(n)$,
and to $1-\tfrac{3}{2} e^{-1/2} \approx 0.09$ when $G=\Alt(n)$.
\end{proof}

\begin{proof}[Proof of Theorem \ref{thSym}]
We prove (\ref{eqSym}) for all $n \geq n_0$ for some integer $n_0$.
Since $r$ is the number of primes dividing $n!$, we actually have $r=\pi(n)$
where $\pi(n)$ is the number of primes less or equal than $n$.
By Proposition \ref{propSym}, we have
$$ \sum_{i=1}^r Prob_{g \in G}(P_i \cap P_i^g \neq 1) < 0.99 + \frac{C \, \pi(n)}{n} $$
for some absolute constant $C>0$.
The proof follows because $\pi(n) = o(n)$.
\end{proof}

We refer the reader to \cite{BH25} for an application of the probabilistic method to non-alternating simple groups.
This requires an analogue of Proposition \ref{propSym} for groups of Lie type \cite[Th. D]{BH25},
which is much more difficult to obtain.

\subsection{Metanilpotent groups} 

We first describe the behaviour of the set $\Gamma_G(P)$ in direct products.
This will produce metabelian groups with small good sets.

\begin{remark} \label{remSmallGS}
Let $G$ be a finite group, $p$ a prime, and let $P \in \Syl_p(G)$.
If $D=G^n$ for some $n \geq 1$, then $Q \in \Syl_p(D)$ is isomorphic to $P^n$.
Moreover, $\Gamma_D(Q)$ is the Cartesian product (as a set) of $n$ copies of $\Gamma_G(P)$ and
$$ \frac{|\Gamma_D(Q)|}{|D|} = \left( \frac{|\Gamma_G(P)|}{|G|} \right)^n . $$
For example, setting $G=\Sym(3)$ and $p=2$,
we have $\Gamma_G(P) = G \setminus P$
and $\frac{|\Gamma_D(Q)|}{|D|} = (2/3)^n \to 0$ for large $n$.
\end{remark}

In the rest of the paper, we will use the following observation without further mention:
If $G$ is metanilpotent and $F(G)=O_p(G)$ for some prime $p$,
then $p$ does not divide $|G:F(G)|$.
The key idea in proving Theorem \ref{thMetaNilp} is to achieve synchronization using a product of commuting elements.
We start with the following technical result.

\begin{lemma} \label{lem:key}
Let $G$ be a group of odd order where $F(G)$ is a Hall subgroup and $G/F(G)$ is nilpotent.
Let $(P_i)_{i=1}^r$ be non-normal Sylow subgroups of $G$ for distinct primes.
If $A \unlhd G$ is abelian and $AN_G(P_i)=G$ for all $i$,
then there exists $x \in A$ such that $P_i \cap P_i^x =1$ for all $i$.
\end{lemma}
\begin{proof}
Let $F=F(G)$, and observe that $A \subseteq F$.
By the Schur-Zassenhaus theorem, there exists $K<G$ such that $G= F \rtimes K$.
We may assume that $\pi(G/F) = \{p_1,\ldots,p_r\}$ and that $P_i \in \Syl_{p_i}(G)$ for each $i$.
For fixed $i$, the hypothesis $A N_G(P_i)=G$ means that $A$ is transitive on $\Syl_{p_i}(G)$ by conjugation.
We first use this fact to show that $F(AK)=A$.

Let $K=K_1 \times \cdots \times K_r$ where $K_i \in \Syl_{p_i}(K) \subseteq \Syl_{p_i}(G)$.
If $g \in F$, then
$$ (AK)^g = AK^g = A(K_1^g \times \cdots \times K_r^g) . $$
By the transitivity of $A$ there exist $a_1,\ldots,a_r \in A$ such that
$$ (AK)^g = A(K_1^{a_1} \times \cdots \times K_r^{a_r}) . $$
Therefore $K_i^{a_i} \subseteq (AK)^g$ for all $i$, and since $a_i \in A$,
also $K_i \subseteq (AK)^g$ for all $i$.
This implies that $K \subseteq (AK)^g$ and so $AK \unlhd G$ because $g$ is arbitrary.
It follows that $F(AK) \subseteq AK \cap F = A$ and equality holds.

By Theorem \ref{thBialoHall} there exists $x \in AK$ such that $K \cap K^x \subseteq K \cap A =1$.
We can assume $x\in A$.
Of course we have
$$ K_i \cap K_i^x \subseteq K \cap K^x = 1 $$
for each $i$.
Using again the transitivity of $A$ on $\Syl_{p_i}(G)$, for each $i$ there exists $b_i\in A$ such that $P_i=K_i^{b_i}$.
But $A$ is abelian and so $b_i x = x b_i$ which gives
$$ P_i \cap P_i^x = K_i^{b_i} \cap K_i^{b_i x} = (K_i \cap K_i^x)^{b_i} = 1 $$
for all $i$, as desired.
\end{proof}

We are ready to deal with the case where $|F(G)|$ and $|G:F(G)|$ are not coprime.
A more involved argument is required to obtain specific elements whose orders satisfy convenient arithmetical conditions.
This is done to apply Lemma \ref{lemIntCen} and then to use a property of centralizers,
namely that $C_G(ab)=C_G(a) \cap C_G(b)$ if $a$ and $b$ are commuting elements of coprime orders.

\begin{proof}[Proof of Theorem \ref{thMetaNilp}]
We work by induction on the order of $G$.
Let $P_i \in \Syl_{p_i}(G)$ for distinct primes $p_1,\ldots,p_n$,
so that we want to find $x \in F(G)$ such that $P_i \cap P_i^x = O_{p_i}(G)$ for all $i$.

We first reduce to the situation where the Frattini subgroup $\Phi(G)$ is trivial.
This is because $\Phi(G/\Phi(G))=1$,
and $F(G/\Phi(G))=F(G)/\Phi(G)$ \cite[Th. 10.6(c)]{DH92}.
In fact $\Phi(G)P_i/\Phi(G) \in \Syl_{p_i}(G/\Phi(G))$,
and working in $G/\Phi(G)$ we obtain $x \in F(G)$ such that $\Phi(G)P_i \cap  \Phi(G)P_i^x \subseteq F(G)$, which implies $P_i \cap P_i^x = O_{p_i}(G)$ for all $i$.

So $\Phi(G)=1$ and $F=F(G)$ is a direct product of elementary abelian groups.
Up to arranging the $p_i$'s, set $p=p_1$ such that $O_p(G) \neq 1$.
If $O_{p'}(F)=1$, then $F=O_p(G)$.
In this case $F$ is an elementary abelian $p$-group and we are done by Lemma \ref{lem:key}. (Set $A=F$ and note that for each $i$, the normalizer $N_G(P_i)$ contains a complement of $F$ in $G$.)

So $O_{p'}(F) \neq 1$, and let $N=O_{p'}(G)$.
Note that $O_{p'}(G/N)=1$ and $F(G/N) = O_p(G/N)$.
Since $G/N$ is metanilpotent, $O_p(G/N) \in \Syl_p(G/N)$.
If $P=P_1$, then $NP/N=O_p(G/N)=F(G/N)$ and $F(NP)=F$.

Let $F = A \times B$ where $A=O_p(G)$ and $B=O_{p'}(F)$.
Note that $AN=A \times N$ and $F(AN)=F$.
Also $F(NP/A)$ is a $p'$-group and is therefore contained in $O_{p'}(NP/A)$.
Now $NA/A \subseteq O_{p'}(NP/A)$,
but the $p'$-part of $|NP/A|$ is $|NA/A|$, so $NA/A = O_{p'}(NP/A)$.
Therefore $F(NP/A) \subseteq NA/A$.
It follows that 
\begin{equation} \label{eqFNPA}
F(NP/A)=F(NA/A) = F/A \cong B . 
\end{equation}

We now observe that
$$ \left( \frac{AN}{N} \right) N_{G/N}\left(\frac{NP_i}{N}\right)
\geqslant
\frac{F N_G(P_i)}{N}
= \frac{G}{N} $$
for all $i$.
The last equality is because $G/F$ is nilpotent and $N_G(P_i)$ contains a complement of $F$ in $G$.
By Lemma \ref{lem:key} applied to $G/N$ and $AN/N$ there exists $a_{_0} \in A$ such that
\begin{equation} \label{eqInsideNP}
 P_i \cap P_i^{a_{_0} b} \subseteq NP_i \cap NP_i^{a_{_0}} \subseteq NP 
\end{equation} 
for all $i$ and all $b \in N$.

For each $i$, let $Q_i = NP \cap P_i$.
Note that $P=Q_1$ and $Q_i \in \Syl_{p_i}(NP)$ for all $i$.
By the inductive hypothesis on $NP/A$, and (\ref{eqFNPA}), there exists $b_{_0} \in B$ such that 
\begin{equation} \label{eqInsideF}
Q_i \cap Q_i^{a b_{_0}} \subseteq AQ_i \cap A Q_i^{b_{_0}} \subseteq F
\end{equation}
for all $i$ and all $a \in A$.

For each $i$, we write $\hat{b}_i$ to denote the $p'_i$-part of $b_{_0}$.
(Note that $\hat{b}_1 = b_0$.)
The $p_i$-part of $b_{_0}$ lies in $Q_i$ and so, by (\ref{eqInsideF}),
\begin{equation} \label{eqInsideFPPart}
Q_i \cap Q_i^{\hat{b}_i} =
Q_i\cap Q_i^{b_{_0}} \subseteq F
\end{equation}
for all $i$. 

Let $x=a_{_0} b_{_0} \in F$, and we will show that $P_i \cap P_i^x = O_{p_i}(G)$ for all $i$.
When $i=1$ (that is $P=P_1=Q_1$) we are done by (\ref{eqInsideF}).
Fix $i \geq 2$.
From (\ref{eqInsideNP}) we have
$$ P_i \cap P_i^x = P_i \cap P_i^x \cap NP = Q_i \cap Q_i^x . $$
Since the $p_i$-part of $b_{_0}$ lies in $Q_i$, we have
$$ Q_i \cap Q_i^x =
Q_i \cap Q_i^{a_{_0} \hat{b}_i} . $$
Now it is useful to work with centralizers.
Looking at $O_{p_i'}(G) \rtimes Q_i$ and using Lemma \ref{lemIntCen},
we obtain $C_{Q_i}(g) =Q_i \cap Q_i^g$ for every $p_i'$-element $g \in G$.
Note that $a_{_0}$, $\hat{b}_i$ and $a_{_0} \hat{b}_i$ are all $p_i'$-elements.
Moreover, since $a_{_0}$ and $\hat{b}_i$ are commuting elements of coprime orders,
we can express both $a_{_0}$ and $\hat{b}_i$ as suitable powers of $a_{_0} \hat{b}_i$.
This implies that $C_{Q_i}(a_{_0} \hat{b}_i) = C_{Q_i}(a_{_0}) \cap C_{Q_i}(\hat{b}_i)$.
We obtain
$$ P_i \cap P_i^x = 
C_{Q_i}(a_0 \hat{b}_i) \subseteq 
C_{Q_i}(\hat{b}_i) =
Q_i \cap Q_i^{\hat{b}_i} . $$
We now complete the proof by appealing to (\ref{eqInsideFPPart}).
\end{proof}

\vspace{0.3cm} \noindent
 {\bfseries Acknowledgments:}
 The second author is supported by the Royal Society.
 We thank S. Eberhard, D. Garzoni, H.Y. Huang and E. Maini for useful conversations,
 and an anonymous referee for a careful reading of the manuscript and several comments.

   \vspace{0.5cm}

\end{document}